\documentclass[12pt]{amsart}
\usepackage{mathptmx}
\usepackage{array}

\usepackage{tikz}
\usepackage{graphics}
\headsep=1cm \numberwithin{equation}{section}
\newtheorem{theorem}{Theorem}[section]
\newtheorem{lemma}[theorem]{Lemma}
\newtheorem{proposition}[theorem]{Proposition}
\newtheorem{corollary}[theorem]{Corollary}

\theoremstyle{definition}
\newtheorem{definition}[theorem]{Definition}
\theoremstyle{remark}
\newtheorem{remark}[theorem]{Remark}
\newtheorem{example}[theorem]{Example}

\newtheorem{acknowledgement}{Acknowledgement}

\newcommand{\reg}{\operatorname{reg}}
\newcommand{\Hreg}{\operatorname{Hreg}}

\newcommand{\Supp}{\operatorname{supp}}
\newcommand{\Tor}{\operatorname{Tor}}

\newcommand{\fm}{\frak{m}}

\begin{document}
\footnotetext{Corresponding author}
\author[Winfried Bruns and Hero Saremi ]{Winfried Bruns and Hero Saremi$^{*}$ }
\title{Hilbert Regularity of Stanley-Reisner Rings}

\address{Universit\"{a}t Osnabruck, Ininstitue f\"{u}r mathematik, 49069 Osnabr\"{u}k, Germany. }  \email{wbruns@uos.de }
\address{Department of Mathematics, Sanandaj Brunch, Islamic Azad University, Sanandaj, Iran.} \email{hero.saremi@gmail.com}



\begin{abstract} 
In this note, we characterize the Hilbert regularity of the Stanley-Reisner ring $K[\bigtriangleup]$ in terms of the $f$-vector and the $h$-vector of a simplicial complex $\bigtriangleup$. We also compute the Hilbert regularity of a Gorenstein algebra.
\end{abstract}

\maketitle


\section{Introduction}
 
Let $K$ be a field and let $M$ be a finitely generated graded module over a standard graded $K$-algebra $R$. The {\it Costelnuovo-Mumford regularity} of $M$ is given by
\begin{equation*}
\reg (M)=\max\{i+j: H^i_{\fm}(M)_{j}\neq 0\},
\end{equation*}
where $\fm$ is the maximal ideal of $R$ generated by the homogeneous elements of positive degree. Let $R=K[X_{1},\cdots, X_d]$ with its standard grading. In this case a theorem of Eisenbud and Goto (see\cite{BH}, 4.3.1) states that
\begin{equation*}
\reg (M)=\max\{j-i: \Tor^R_{i}(K, M)_{j}\neq 0\},
\end{equation*}
where $K$ is naturally identified with $R/{\fm}$.

Bruns, Moyano and Uliczka, in \cite{BMU}, defined the $Hilbert\ regularity$ for a $\mathbb{Z}$-graded module $M$ over a polynomial ring as the least possible value of $\reg N$ for a module $N$ with $H_M(t)=H_N(t)$. Then they introduced $(n,k)${\it-boundary presentations} as the basic tool for the analysis of Hilbert series and the computation of Hilbert regularity. We recall their definition:

\begin{definition}\label{boundary}
Let $H_{M}(t)=\dfrac{Q(t)}{(1-t)^d}$ be the Hilbert series of $M$ with $d=\dim(M)$ and $Q(t)\in \mathbb{Z}[t]$. For integers $n$, $0\leq n\leq d$, and $k\geq 0$, an $(n, k)$-\emph{boundary presentation} of $H_M(t)$ is a decomposition of $H_M(t)$ in the form
\begin{equation*}
H_M(t)=\sum_{i=0}^{k-1}\dfrac{f_it^i}{(1-t)^n}+\dfrac{ct^k}{(1-t)^n}+\sum_{j=0}^{d-n-1}\dfrac{g_{j}t^k}{(1-t)^{d-j}}
\end{equation*} with $f_i, c,g_i \in \mathbb{Z}$. If $c=0$, the boundary presentation is called \emph{corner-free}, and if the coefficients $f_i, c,g_j$ are nonnegative, the boundary presentation is called \emph{nonnegative}.
\end{definition}

In the following we will simply speak of an $(n,k)$-\emph{presentation}. An $(n,k)$-presentation of  $H_{M}(t)={Q(t)}/{(1-t)^d}$ exists if and only if $k-n\geq \deg H=\deg Q-d$. If it exists, it is uniquely determined (see \cite{BMU}, Corollary 3.4).

In this paper we assume that modules $M$ are generated in nonnegative degrees. This is not a severe restriction since it can always be achieved by passing to a shifted version $M(-s)$ with $s\ge 0$.

In \cite[Cor. 4.2]{BMU} Bruns at el.\ characterized Hilbert regularity in terms of boundary presentations as follows:
\begin{theorem}
\begin{enumerate}
\item	$\Hreg H\ge k$ (if and) only if $H_M(t)$ admits a nonnegative $(0,k)$-boundary presentation. 
\item If $H_M(t)$ admits a non-corner-free $(0,k)$-boundary presentation, then $\Hreg H \geq k$.  
\end{enumerate}
\end{theorem}

A different definition of Hilbert regularity was given by Herzog \cite{H}, but the two definitions are equivalent, as follows from the description in terms of boundary presentations.
 
The first goal of this note is to estimate the Hilbert regularity of the Stanley-
Reisner ring $K[\bigtriangleup]$ of a simplicial complex $\bigtriangleup$ in terms of its $f$-vector and $h$-vector. Then, we compute the Hilbert regularity of  boundary complexes of simplicial polytopes  and, more generally, of Gorenstein algebras. 

For the notions of commutative algebra we refer the reader to Bruns and Herzog \cite{BH} and Stanley \cite{S}.

\section{Hilbert regularity and $f$-vector }

The term ``boundary presentation'' of \cite{BMU} is motivated by a visualization of a decomposition of a Hilbert series: A decomposition
\begin{equation*}
	\dfrac{Q(t)}{(1-t)^d}=\sum_{i=0}^d\sum_{j\geq 0} a_{ij}\dfrac{t^j}{(1-t)^i}
\end{equation*}
can be depicted on a square grid with the box at position $(i,j)$ labeled by $a_{ij}$. 

In the case of an $(n,k)$-presentation, the nonzero labels in this grid form the bottom and the right edge of a rectangle with $d-n+1$ rows and $k+1$ columns. The coefficient in the ``corner'' plays a double role since it belongs to both edges, therefore it is denoted by an extra letter. In the case $n=0$, the only one relevant for us, one obatins the figure
\newlength\savewidth
\newcolumntype{I}{!{\vrule width 1pt}}
\newcolumntype{I}{!{\vrule width 1pt}}
\newcommand{\whline}{\noalign{\global\savewidth\arrayrulewidth%
\global\arrayrulewidth1pt}%
\hline %
\noalign{\global\arrayrulewidth\savewidth}}	
\begin{center}
\begin{tabular}{cIc|c|c|cIc|c}
$(1-t)^{-d}$& & &  & $g_0$ &   \\
\hline $\vdots$ & & & & $\vdots$&  \\
\hline $(1-t)^{-1}$ & &  &&$g_{d-1}$&  \\
\hline $1$ &$f_0$ &$\cdots$ & $f_{k-1}$ &$c$ && \\ 
\whline  & $1$ & $\cdots$ & $t^{k-1}$ & $t^k$& $t^{k+1}$&$\cdots$
\end{tabular} 
\end{center}

In order to obtain a boundary presentation, one must usually ``move'' entries of this grid that are not yet at the desired position. The basic rule for moving terms right and down is  the relation
\begin{equation}
\frac{t^u}{(1-t)^v}=\dfrac{t^{u+1}}{(1-t)^{v}}+\frac{t^{u}}{(1-t)^{v-1}};\label{right}
\end{equation}
Rewriting it, we get the rule for moves to the left and diagonally down:
\begin{equation}
\frac{t^u}{(1-t)^v}=\dfrac{-t^{u-1}}{(1-t)^{v-1}}+\frac{t^{u-1}}{(1-t)^v}.\label{left}
\end{equation}

Clearly $(0,k)$-presentations are additive for linear combinations of the rational functions $\dfrac{t^u}{(1-t)^v}$. Therefore we must understand their boundary presentations, of which we will make extensive use in the rest of the paper. 

\begin{lemma}\label{lright}
Let $k\geq u$. Then the $(0,k)$-presentation of the rational function  $\dfrac{t^u}{(1-t)^v}$ is given by 
\begin{equation}
\dfrac{t^u}{(1-t)^v}=\sum_{n=0}^{k-u-1}\binom{n+v-1}{v-1}t^{u+n}+\sum_{n=1}^v \binom{k-u+v-n-1}{v-n}\dfrac{t^k}{(1-t)^n}.\label{right-k}
\end{equation}

\end{lemma}
\begin{proof}
We use induction on $k$. For $k=u$ the equation is a tautology. So let $k>u$. 

By the induction hypothesis the $(0,k-1)$-presentation of $t^u/(1-t)^v$ is given by
\begin{equation*}
\dfrac{t^u}{(1-t)^v} =\sum_{n=0}^{k-u-2}\binom{n+v-1}{v-1}t^{u+n}+\sum_{n=1}^v \binom{k-u+v-n-2}{v-n}\dfrac{t^{k-1}}{(1-t)^n}.
\end{equation*}

The coefficients of the terms $t^{u+n}$, $0\le n\le k-2$, in the $(0,k)$-presentation are the same as those in the $(0,k-1)$-presentation. By equation \eqref{right} it is clear that the coefficient of $t^{k-1}$ in the $(0,k)$-presentation is the sum of the coefficients of the terms $t^{k-1}/(1-t)^n$ in the $(0,k-1)$-presentation:
$$
\sum_{n=1}^v \binom{k-u+v-n-2}{v-n}=\binom{k-u+v-2}{v-1}
$$
by a standard identity of binomial coefficients. Similarly, the coefficient of $t^k/(1-t)^n$ is
\begin{equation*}
\sum_{m=n}^v \binom{k-u+v-m-2}{v-m}=\binom{k-u+v-n-1}{v-n}.\qedhere
\end{equation*}
\end{proof}

\begin{lemma}\label{lleft}
Let $ u-v\leq k\leq u$. Then the $(0,k)$-presentation for the rational function  $\dfrac{t^u}{(1-t)^v}$ is given by 
\begin{equation}
\dfrac{t^u}{(1-t)^v}=\sum_{n=v+k-u}^{v}(-1)^{v-n}\binom{u-k}{u-k-(v-n)}\dfrac{t^k}{(1-t)^{n}}.
\end{equation}
\end{lemma}

\begin{proof}
We use decreasing induction to proof the result. For $k=u$ we have nothing to prove. For $k<u$ we get 
\begin{multline*}
(-1)^{v-n}\binom{u-(k+1)}{u-(k+1)-(v-n)}-(-1)^{v-(n+1)}\binom{u-(k+1)}{u-(k+1)-(v-(n+1))}\\
=(-1)^{v-n}\binom{u-k}{u-k-(v-n)}
\end{multline*}
by the induction hypothesis and \eqref{left}.\end{proof}

Now we recall the definition of the Stanley-Reisner ring associated with a simplicial complex and study the Hilbert regularity of this ring.
\begin{definition}(\cite{BH}, Definition 5.1.2)
Let $\bigtriangleup$ be a simplicial complex on the vertex set $V$. The Stanley-Reisner ring of the complex $\bigtriangleup$ is the $K$-algebra
$$K[\bigtriangleup]=K[X_1,\cdots,X_n]/I_{\bigtriangleup},$$
where $I_{\bigtriangleup}$ is the ideal generated by all monomials $X_{i_1}X_{i_2}\cdots X_{i_s}$ such that $\{v_{i_1},\cdots v_{i_s}\} \notin
\bigtriangleup$.
\end{definition}

Let $\bigtriangleup$ be a simplicial complex with $f$-vector $(f_0, f_1,\cdots,f_{d-1})$. Then the Hilbert series of $K[\bigtriangleup]$ is
\begin{equation}
H_{K[\bigtriangleup]}(t)=\sum_{i=-1}^{d-1}\dfrac{f_it^{i+1}}{(1-t)^{i+1}}.\label{fH}
\end{equation}

This formula has the grid presentation
\begin{center}
\begin{tabular}{cIc|c|c|c|c}
$(1-t)^{-d}$& & & &$\cdots$ & $f_{d-1}$\\
\hline $\vdots$ & & & & $\dots$ \\
\hline $(1-t)^{-2}$& &  & $f_1 $ &\\
\hline $(1-t)^{-1}$ & & $f_0$ & & \\
\hline $1$ & $1$&  & & \\ 
\whline  & $1$ & $t$ & $t^2$ & $\cdots$& $t^{d}$ 
\end{tabular} 
\end{center}

The goal is to push the entries of the table to the left as long as one keeps nonnegative entires. The basic tool is equation \eqref{left}. If we apply it successively to the table above, pushing the terms in the columns $t^d,t^{d-1},\dots,t^{k+1}$ into the column of $t^{k}$, we obtain new entries $a^{(k)}_i$ for $i=k,\cdots, d$ in the column $t^k$. They are given by the following theorem.

\begin{theorem}\label{th}
Let $\bigtriangleup$ be a $(d-1)$-dimensional simplicial complex with $f$-vector $(f_0, f_1,\cdots, f_{d-1})$. Then $\Hreg(K[\bigtriangleup])\leq k$ if and only if $a_i^{(k)}\geq 0$ for all $i=k,\dots, d,$ where
\begin{equation}
a_i^{(k)}=\sum_{n=i}^{d}(-1)^{n-i} \displaystyle{n-k\choose i-k}f_{n-1}.
\end{equation}

\end{theorem}

\begin{proof}
Note that, by Lemma \ref{lleft}, $\dfrac{f_{n-1}t^n}{(1-t)^n}$ contributes $(-1)^{n-i} \displaystyle{n-k\choose i-k}f_{n-1}$ for $i=k,\cdots ,n$. 
Therefore $a_i^{(k)}=\sum_{n=i}^{d}(-1)^{n-i} \displaystyle{n-k\choose i-k}f_{n-1}$.

On the other hand, the terms $\dfrac{f_{i-1}t^i}{(1-t)^i}$, $i=0,\dots,k-1$, contribute only nonnegative coefficients to the other entries in the $(0,k)$-presentation. Therefore the theorem follows from \cite{BMU}, Corollary 3.4 and Corollary 4.2.
\end{proof}

\begin{example}
Let $\Delta$ be the boundary of the tetrahedron. Then its Hilbert series is $H(t)=1+\frac{4t}{(1-t)}+\frac{6t^2}{(1-t)^2}+\frac{4t^3}{(1-t)^3}$ . Its $(0,2)$-presentation is
\begin{equation*}
	H(t)=1+4t+\frac{4t^2}{(1-t)}+\frac{2t^2}{(1-t)^2}+\frac{4t^2}{(1-t)^3},
\end{equation*}
and the Hilbert regularity is $2$.
\end{example}

In general, the Hilbert regularity is a lower bound for the Costelnuovo-Mumford regularity, therefore the previous theorem has an immediate consequence:

\begin{corollary}
	Let $\bigtriangleup$ be a $(d-1)$-dimensional simplicial complex with $f$-vector $(f_0, f_1,\cdots,\allowbreak f_{d-1})$. If $\reg(K[\bigtriangleup])\leq k$, then $a_i^{(k)}\geq 0$ for all $k\leq i\leq d,$ where $a_i^{(k)}$ are the coefficients in Theorem \ref{th}.
\end{corollary}

\begin{example}
It is easy to see that $\Hreg(K[\Delta])=0$ if and only if $\Delta$ is the full simplex so that $\Hreg(K[\Delta])=0$ if and only if $\reg(K[\Delta])=0$. In general however,  $\Hreg(K[\Delta])<\reg(K[\Delta])$. If $\Delta$ is the graph consisting of the $3$ edges of a triangle, then $\reg(K[\Delta])=2$, but $\Hreg(K[\Delta])=1$.

More generally, let $G$ be a connected graph with at least $2$ edges. Then $K[G]$ is Cohen-Macaulay, and $\reg(K[G])\le 2$. Let $r=\reg(K[G])$ and $h=\Hreg(K[G])$. Then (i) $r=1$ and $h=1$ if $G$ is a tree, (ii) $r=2$ and $h=1$ if $G$ is a cycle, (iii) $r=2$ and $h=2$ in all other cases.
\end{example}

\begin{remark}\label{sq}
Theorem \ref{th} only requires that the Hilbert series of $R$ can be written in the form of equation \eqref{fH} with nonnegative coefficients $f_i$. A typical class of such examples are the \emph{squarefree modules} over a polynomial ring introduced by Yanagawa in \cite{Y}. A Stanley-Reisner ring $K[\bigtriangleup]$ is a squarefree $R$-module. We recall the definition.

Let $R=K[x_1,\cdots, x_n]$ be a polynomial ring. Consider the natural ${\mathbb{N}}^n$-grading on $R$. For ${\bf a}=(a_1,\cdots, a_n)\in {\mathbb{Z}^n}$, we set $\Supp( {\bf a}):=\{i \  | \ a_i\neq 0\}\subset [n]=\{1,\cdots,n\}$ and ${|\bf a|}=a_1+\cdots +a_n$. A finitely generated $\mathbb{N}^n$-graded $R$-module $M$ is called \emph{squarefree} if the multiplication map
$$
\times x_i :M_{\bf a} \rightarrow M_{{\bf a}+{\bf e}_i}
$$
is bijective for all ${\bf a}\in \mathbb{N}^n
$ and all $i\in \Supp(\bf a)$.

 If $M$ is squarefree, then $\dim_K M_{\bf a}=\dim_K M_{\Supp({\bf a})}$ for all ${\bf a} \in \mathbb{N}^n$, and $M$ is generated by its squarefree part $\{M_{\sigma}\ |\ \sigma\subset [n]\}$. Moreover, the Krull dimension of $M$ is $\max \{|\sigma| \ \ | M_{\sigma}\neq 0\}$.
 
 Let $M$ be a squarefree module of dimension $d$. The $\mathbb{Z}^n$-graded Hilbert series of $M$ is 
\begin{equation*}
H_M(t_1,\cdots t_n)=\sum_{\sigma\subseteq [n]}\dim_K M_\sigma\prod_{i\in \sigma} \dfrac{t_i}{1-t_i}
\end{equation*}
and it specializes to the $\mathbb{Z}$-graded Hilbert series
\begin{equation*}
H_M(t)=\sum_{\sigma\subseteq [n]}{\dim_K M_\sigma } \dfrac{t^{|\sigma|}}{(1-t)^{|\sigma|}}.
\end{equation*}  
Set $f_i=\sum\limits_{\sigma\subseteq [n],|\sigma|=i} \dim_KM_{\sigma}$, where $f_\emptyset (M)=h_\emptyset(M)=\dim _KM_{\bf 0}$.
Then 
\begin{equation}
H_M(t)=\sum_{i=0}^d \dfrac{f_{i-1} t^i}{(1-t)^i}
\label{sqh}
\end{equation}
Via this presentation of the Hilbert series one can generalize Theorem \ref{th} to squarefree modules.

Note that the Hilbert series that are given by  \eqref{sqh} with nonnegative coefficients $f_i$ are exactly the Hilbert series of the squarefree modules
 \begin{equation*}
 M=\bigoplus_{i=0}^d \bigoplus_{j=1}^{f_{i-1}}(x_1\cdots x_i) K[x_{1},\cdots,x_i].
 \end{equation*} 

\end{remark}

\section{Hilbert regularity and $h$-vector}
A  graded module $M\neq 0$ of dimension $d$ has a Hilbert series of the form $H_M(t)=\dfrac{Q_M(t)}{(1-t)^d}$ where $Q_M(t)$ is polynomial with integer coefficients. Therefore we can write 
\begin{equation*}
H_{M}(t)=\dfrac{h_0+h_1t+\cdots+h_st^s}{(1-t)^d}, \quad h_s\neq 0.
\end{equation*}
Recall that we assume $M$ is finitely generated over a standard graded $K$-algebra and is generated in nonnegative degrees. 
The finite sequence of integers $h(M)=(h_0,h_1,\cdots,h_s)$ is called the {\it h-vector} of $M$. The $h$-vector of a simplicial complex is the $h$-vector of its Stanley-Reisner ring. In the following proposition we give a $(0,k)$-presentation of the Hilbert series in terms of the $h$-vector of $M$.

Before we formulate the proposition let us have a look at the grid presentation in terms of the $h$-vector. The entries $*$ mark the squares for the $(0,k)$-presentation:
\begin{center}
	\begin{tabular}{cIc|c|c|cIc|c|c}
		$(1-t)^{-d}$& $h_0$ & $\cdots$& $h_{k-1}$ & $h_k$ & $h_{k+1}$ & $\cdots$ & $h_s$\\
		\hline $\vdots$ & & & &$\vdots$&&&  \\
		\hline $(1-t)^{-1}$ & &  & &*&&&  \\
		\hline $1$ & *&$\cdots$  &*&*&&& \\ 
		\whline  & $1$ & $\cdots$ & $t^{k-1}$ & $t^k$& $t^{k+1}$ &$\cdots$ & $t^s$
	\end{tabular} 
\end{center}
It is clear that the entries in the columns  $0,\dots,k-1$ must be pushed successively to the right until we have entries only in row $0$ and column $k$. Those in columns $k+1,\dots,s$ must be successively pushed to the left until all terms have accumulated in column $k$. The hypothesis on $k$ in the proposition ensures that we do not reach row $0$ before column $k$.
 
\begin{proposition}\label{pro}
Let $h(M)=(h_0,h_1,\cdots,h_s)$ be the $h$-vector of the graded module $M$. If $k\geq s-d$, then the  $(0,k)$-presentation exists and is given by
\begin{equation*}
H_M(t)=\sum_{i=0}^{k-1}x_i^{(k)}t^i+\sum_{j=1}^{d}\dfrac{y_j^{(k)}t^k}{(1-t)^j}+\sum_{r=k+d-s}^{d}\dfrac{z_r^{(k)}t^k}{(1-t)^{r}}
\end{equation*}
with the coefficients
\begin{align*}
x_i^{(k)}=&\sum_{n=0}^{i}\binom{d+i-n-1}{d-1}h_n, & for\ \ i=0,\cdots,k-1,\\
y_j^{(k)}=&\sum_{n=0}^{k-1}\binom{k-n+d-j-1}{d-j}h_n, & for \ \ j=1,\cdots,d,\\
z_r^{(k)}=&(-1)^{d-r}\sum_{n=0}^{r+s-d-k}\binom{s-k-n}{d-r}h_{s-n},& for\ \  r=k+d-s,\cdots,d.
\end{align*}
\end{proposition}

\begin{proof}
Note that 
\begin{equation*}
H_M(t)=\dfrac{h_0+h_1t+\cdots+h_st^s}{(1-t)^d}=\sum_{n=0}^{k-1}\dfrac{h_nt^n}{(1-t)^d}+\sum_{n=k}^{s}\dfrac{h_nt^n}{(1-t)^d}.
\end{equation*}
Apply Lemma \ref{lright} to the first summand and Lemma \ref{lleft} to the second one.
\end{proof}

The $f$-vector and the $h$-vector of a $(d-1)$-dimensional simplicial complex $\bigtriangleup$ are related by
\begin{equation*}
\sum_ih_it^i=\sum_{i=0}^d
f_{i-1}t^i(1-t)^{d-i}.
\end{equation*}
In particular, the $h$-vector has length at most $d$, and, for $j=0,\cdots,d,$
\begin{equation}
h_j=\sum_{i=0}^{j}(-1)^{j-i}\displaystyle{d-i\choose j-i}f_{i-1}
\quad \text{and}\quad f_{j-1}=\sum_{i=0}^j\displaystyle{d-i\choose j-i}h_i.
\label{hf}
\end{equation}

The following corollary is a consequence of Proposition \ref{pro} and Theorem \ref{th}.

\begin{corollary}
Let $\bigtriangleup$ be a $(d-1)$-dimensional simplicial complex with the $h$-vector $(h_0, h_1,\cdots, h_s)$, and set $h_{s+1}=\cdots =h_d=0$. Then $\Hreg(K[\bigtriangleup])\leq k$ if and only if
\begin{equation}
\sum_{n=0}^{k-1} \binom{k-n+d-i-1}{d-i}h_n+(-1)^{d-i}\sum_{n=0}^{i+s-d-k} \binom{s-k-n}{d-i}h_{s-n}\ge 0
\label{hv}
\end{equation}
for $i=k+d-s\dots,d$ .
\end{corollary}

\begin{proof}
Note that $a_i^{(k)}=y_i^{(k)}+z_i^{(k)}$ for $i=k+d-s,\cdots, d$ and apply Theorem \ref{th}.
\end{proof}

By formula \eqref{hf} one can conversely derive Theorem \ref{th} from the above corollary.

We can give a very simple description of the Hilbert regularity of a Gorenstein algebra $A$. Let $H_A(t)=(1-t)^{-d}(h_0+h_1t+\cdots+h_st^s)$. If $A$ is Gorenstein, then $h_i\ge 0$ and $h_{s-i}=h_i$ for all $i=0,\dots,s$. 

\begin{corollary}\label{Gor}
Let $(h_0,\cdots, h_s)$  be the $h$-vector of the Gorenstein algebra $A$.
\begin{itemize}
\item[(i)]
If $s<2d$, then $\Hreg (A)= \lceil\frac{s}{2}\rceil$; 
\item[(ii)]
if $ s\geq 2d$, then
\begin{equation*}
\Hreg (A)=
\begin{cases}
s-d,& d\text{ even}, \\
s-(d-1),& d\text{ odd}. 
\end{cases}
\end{equation*}  
\end{itemize}
\end{corollary}

\begin{proof}
Let $s<2d$. Then $k=\lceil\dfrac{s}{2}\rceil>s-d$, and there exists a corner-free $(0,k)$-presentation for $H_A(t)$. With the notation of Proposition \ref{pro}, we show that $y_j+z_{j}$ is nonnegative for $j=d+k-s,\dots, d$. By Proposition \ref{pro} 
\begin{equation*}
y_j^{(k)}+z_{j}^{(k)}=\sum_{n=0}^{k-1}\binom{k-n+d-j-1}{d-j}h_n+(-1)^{d-j}\sum_{n=0}^{j+s-d-k}\binom{s-k-n}{d-j}h_{s-n},
\end{equation*}
and by the symmetry of the $h$-vector, 
\begin{align*}
y_j^{(k)}+z_{j}^{(k)}&=\sum_{n=0}^{j+s-d-k}\biggl(\binom{k-n+d-j-1}{d-j}+(-1)^{d-j}\binom{s-k-n}{d-j}\biggr)h_{n}\\
&\qquad +\sum_{n=j+s-d-k+1}^{k-1}\binom{k-n+d-j-1}{d-j}h_n.
\end{align*}
Set $i=j+\lfloor\frac{s}{2}\rfloor -d$, then
\begin{align*}
y_i^{(k)}+z_{i}^{(k)}&= \sum_{n=0}^{i}\biggl(\binom{s-n-i-1}{\lfloor\frac{s}{2}\rfloor -i}+(-1)^{\lfloor\frac{s}{2}\rfloor -i}\binom{\lfloor\frac{s}{2}\rfloor -n}{\lfloor\frac{s}{2}\rfloor -i}\biggr)h_{n}\\
&\qquad +  \sum_{n=i+1}^{\lceil\frac{s}{2}\rceil -1}\binom{s-n-i-1}{\lfloor\frac{s}{2}\rfloor -i}h_n.
\end{align*}
The coefficients in the first summand of the above equation are nonnegative for $i=0,\cdots, \lfloor\frac{s}{2}\rfloor$. Therefore $\Hreg(A)\leq \lceil\frac{s}{2}\rceil$.
On the other hand 
\begin{align*}
y_{d-1}^{(k-1)}+z_{d-1}^{(k-1)}&=
\sum_{n=0}^{k-2}{(k-n-1)}h_n-\sum_{n=0}^{s-k}(s-k-n+1)h_{s-n}\\
&=(2k-s-2)\sum_{n=0}^{k-2}h_n-\sum_{n=k-1}^{s-k}(s-k-n+1)h_{s-n}\\
&=(2k-s-2)\sum_{n=0}^{s-k}h_n-\sum_{n=k-2}^{s-k}(s-k-n+1)h_{s-n}
\end{align*}
is negative for $k=\lceil\frac{s}{2}\rceil$. Therefore $\Hreg(A)\geq \lceil\frac{s}{2}\rceil $, and the equality holds.

Let $s\geq 2d$; then by the same argument as in the first part we can see that $\Hreg (A)\leq s-d$ for an even number $d$ and $\Hreg (A)\leq s-(d-1)$ for an odd number $d$. On the other hand, by Lemma 3.8 of \cite{BMU}, there does not exist a $(0, s-d-1)$-presentation for $H_A(t)$. Thus $\Hreg (A)\geq s-d$. For odd $d$ the coefficient $z_0=(-1)^dh_s<0$ in the corner of the $(0,s-d)$-presentation is negative. Therefore  $\Hreg (A)\geq s-(d-1)$ in this case.
\end{proof}

Corollary \ref{Gor} only uses the equation $h_i=h_{s-i}$ for $i=0,\dots,s$ and the nonnegativity of the $h_i$. 

Noteworthy examples of Gorenstein rings are the Stanley-Reisner rings of boundary complexes of simplicial polytopes. Their $h$-vector has length $d$, and since they are Cohen-Macaulay, they have Castelnuovo-Mumford regularity $d$.

\begin{corollary}
Let $\bigtriangleup _{\mathcal{P}}$ be the boundary complex of the simplicial $d$-polytope $\mathcal{P}$ with $h$-vector $(h_0, h_1,\cdots, h_d)$. Then $\Hreg(K[\bigtriangleup _{\mathcal{P}}])= \lceil\frac{d}{2}\rceil$.
\end{corollary}

\begin{acknowledgement}
The second author would like to thank the Department of Mathematics of the University of Osnabr\"uck  for hospitality and partial support during the time that she was spending her sabbatical semester in this university. Her work was also partially supported by a grant from the Simons Foundation. The authors thank the referee for his carefully reading.  
\end{acknowledgement}


\end{document}